\documentclass[12pt,a4paper]{article}
\usepackage[utf8]{inputenc}
\usepackage[T1]{fontenc}
\usepackage{amsmath}
\usepackage{amsthm}
\usepackage{amsfonts}
\usepackage{amssymb}
\usepackage{graphicx}
\usepackage{authblk}
\usepackage{hyperref}
\usepackage{tikz}
\usepackage{wrapfig}
\usepackage{caption}
\usepackage{subcaption}
\usepackage[capposition=bottom]{floatrow}
\usepackage[document]{ragged2e}
\usepackage[left=2.50cm, right=2.50cm, top=2.00cm, bottom=2.00cm]{geometry}
\usepackage{xcolor}
\usepackage[backend=biber,style=alphabetic]{biblatex}
\newtheorem{theorem}{Theorem}[section]
\newtheorem{lemma}[theorem]{Lemma}

\newtheorem{proposition}[theorem]{Proposition}
\newtheorem{corollary}[theorem]{Corollary}

\newtheorem{example}[theorem]{Example}
\newtheorem{conjecture}[theorem]{Conjecture}

\theoremstyle{definition}
\newtheorem{definition}[theorem]{Definition}

\usepackage{ytableau} 
\usepackage{authblk}
\usepackage{multirow}
\usepackage{placeins}

\newcommand{\BLUE}[1]{{\color{blue}#1}}
\newcommand{\RED}[1]{{\color{red}#1}}
\newcommand{\Z}{\mathbb{Z}}
\newcommand{\N}{\mathbb{N}}
\newcommand{\T}{\mathcal{T}}
\newcommand{\instar}{\operatorname{instar}}

\title{\textbf{LABELED CHIP-FIRING ON STAR GRAPHS}}

\author[1]{Annika Gonzalez-Zugasti}
\author[1]{Ryan Lynch}
\author[2]{Dylan Snustad}
\affil[1]{University of Minnesota - Twin Cities}
\affil[2]{Arizona State University}
\date{\today}

\begin{document}
\setlength{\abovedisplayskip}{10pt}
\setlength{\belowdisplayskip}{10pt}
\setlength{\abovedisplayshortskip}{10pt}
\setlength{\belowdisplayshortskip}{10pt}
	
	\maketitle
	
	\begin{abstract}
	\justify
	We study the stable configurations of the labeled chip-firing game on an infinitely subdivided $k$-star graph starting with $km$ chips on the center vertex. We prove a sorting property of this game and analyze special stable configurations corresponding to standard Young tableaux.
		
	\end{abstract}
	
	\section{Introduction}\label{sec:intro}
	
	\justify
	In the traditional chip-firing game, identical ``chips'' are placed on each vertex of a graph. If a vertex has at least as many chips as neighbors, it can ``fire'' by sending one chip to each neighbor. The game ends when no vertices have enough chips to fire. In this case, the resulting arrangement of chips on the graph is called a stable configuration. This discrete dynamical system has connections to other areas of combinatorics, algebra, and physics. 

    Labeled chip-firing is a variation of this game, due to \cite{Hopkins_2017}, that assigns labels $1$ through $N$ for each chip on the graph and specifies additional rules for how these labeled chips must move during a vertex fire. In the most well-known version of this game, $N$ chips are placed on a single vertex in an infinite line graph. When a vertex fires in this setup, the smaller chip must move to the left and the larger chip must move to the right. In \cite{Hopkins_2017} it was shown that when starting with $2m$ chips, the resulting stable configuration will always be the same, and will result in the chips labeled $1$ through $2m$ being placed in order from left to right on the graph.
    
    In 2022, Musiker and Nguyen proved a similar sorting property for the case where the base graph is a binary tree with $2^n-1$ labeled chips starting on the origin:

    \begin{theorem}[{\cite[Theorem 1.2]{musiker2023labeledchipfiringbinarytrees}}]
        In labeled chip-firing on an infinite binary tree with $2^n-1$ chips (labeled from $1$ to $2^n-1$) initially placed at the origin, the terminal configuration always has one chip at each node of the first $n$ levels. Moreover, the bottom straight left and right descendants of a node contain the smallest and largest chips among its subtree.
    \end{theorem}

   This paper's format and terminology are inspired by Musiker and Nguyen's work; particularly in our use of the term ``endgame fires,'' which we define in Section \ref{subsec:main-theorem}.
   Recently, in September 2025, Inagaki and Lin \cite{inagaki2025labeledchipfiringundirectedkary} worked to generalize this concept to $k$-ary trees and were able to provide both upper and lower bounds on the number of unique stable configurations arising from the configuration beginning with $\frac{k^\ell-1}{k-1}$ (for $\ell\in\Z_{>0}$) chips on the origin. 

    
    In this paper we introduce a variation of this idea by playing the labeled chip-firing game on infinitely subdivided star graphs. We prove that each ``branch'' will be ordered when one starts with $km$ chips on the center vertex of a star graph with $k$ branches. This sorting condition is weaker than that shown in \cite{Hopkins_2017}, but allows for a richer analysis of the types of stable configurations that do occur. We also show that for the $m=2$ case, there is a bijection between stable configurations and $k\times 2$ standard Young tableaux, and that for $m>2$ a bijection exists between $k\times m$ SYTs and stable configurations with an additional ``volatility minimizing'' condition.

    \vspace{1em}
    
    \noindent \textbf{Acknowledgements:} The authors thank Gregg Musiker for his invaluable guidance and support as the advisor of this project. Dylan Snustad also thanks Vic Reiner and NSF Grant DMS-$2053288$ for partial support during Summer 2024.

	\justify
	\section{A Review of Unlabeled Chip-Firing}\label{sec:unlabeled-gen}
	
	\justify
	We now provide a review of classical chip-firing. Begin with an undirected graph $G=(V,E)$ with $|V(G)|=n$. We say a vector $C=(C(v_1),\ldots,C(v_n))\in \N^n$ is a \textit{chip configuration} (we use $\N=\Z_{\geq 0}$), and a vertex $v_i$ has $C(v_i)$ chips on it. We let $N$ denote the total number of chips in a configuration, so $N=\sum_{v\in V(G)}C(v)$. If $C(v)\geq\deg(v)$, then $v$ is \textit{ready to fire}. Vertices are fired one at a time. Firing vertex $v$ sends one of its chips to each of its neighboring vertices along an edge, resulting in a new chip configuration $C'$ with $C'(v)=C(v)-\deg(v)$, $C'(u)=C(u)+1$ if $e_{vu}\in E(G)$, and $C'(u)=C(u)$ otherwise. We say a chip configuration $C'$ is \textit{reachable from $C$} if there is a sequence of vertex firings beginning with $C$ and ending in $C'$. A chip configuration is \textit{stable} if no vertex is ready to fire. The following theorem is a well-known result detailing which configurations on a given graph $G$ have reachable stable configurations.

    \begin{theorem}[\cite{BJORNER1991283}]
        Let $N$ be the total number of chips, i.e. $N = \sum_{v\in V(G)}C(v)$.
        \begin{enumerate}
            \item[(a)] If $N>2|E(G)|-n$ then the game will never stabilize.
            
            \item[(b)] If $N<|E(G)|$ then the configuration reduces to a unique stable configuration.
            
            \item[(c)] If $|E(G)|\leq N\leq 2|E(G)|-n$ then there exists some initial chip configurations which leads to an infinite process, and some other initial configurations which terminates in finite time.
        \end{enumerate} 
    \end{theorem}
    We will be concerned with finite configurations on graphs with infinitely many vertices and edges,  so the process will always terminate.

    \pagebreak
    
    \justify
    \section[]{Unlabeled Chip-Firing on the Infinite $k$-Star}\label{sec:unlabeled-star}

    Let an \textit{infinite $k$-star graph}, ${\operatorname{instar}(k)}$, refer to an undirected graph consisting of $k$ \textit{branches}, each of which have infinitely many vertices connected in a line. All $k$ branches share one central vertex, the \textit{center} of the infinite $k$-star. Label the center by ${v_{i,0}}$ for all $i\in[k]:=\{1,2\ldots,k\}$. In other words, $v_{i,0}$ refers to the unique center vertex which is shared by all the branches, no matter which branch's index we use to refer to it. Label the vertices on the branches by ${v_{i,j}}$ for $(i,j)\in[k]\times\Z_{>0}$ where $j$ is the distance from the center to the vertex in question. See Figure \ref{fig:star-example} for an example. We call $j$ the \textit{level} of vertex $v_{i,j}$.
    
    \begin{figure}[ht]
        \centering
        \begin{tikzpicture}[scale=0.5]
            \node[circle,fill=black,scale=0.5,label={375:$v_{i,0}$}] at (360:0mm) (center) {};
            \node at (90:8cm) (n1) {branch 1};
            \node[circle,fill=black,scale=0.5,label={360:$v_{1,1}$}] at (90:2cm) (m1) {};
            \node[circle,fill=black,scale=0.5,label={360:$v_{1,2}$}] at (90:4cm) (l1) {};
            \node[circle,fill=black,scale=0.5,label={360:$v_{1,3}$}] at (90:6cm) (p1) {};
            \draw[->, thick] (center)--(n1);
            \node at (330:8cm) (n2) {branch 2};
            \node[circle,fill=black,scale=0.5,label={-120:$v_{2,1}$}] at (330:2cm) (m2) {};
            \node[circle,fill=black,scale=0.5,label={-120:$v_{2,2}$}] at (330:4cm) (l2) {};
            \node[circle,fill=black,scale=0.5,label={-120:$v_{2,3}$}] at (330:6cm) (p2) {};
            \draw[->, thick] (center)--(n2);
            \node at (210:8cm) (n3) {branch 3};
            \node[circle,fill=black,scale=0.5,label={90:$v_{3,1}$}] at (210:2cm) (m3) {};
            \node[circle,fill=black,scale=0.5,label={90:$v_{3,2}$}] at (210:4cm) (l3) {};
            \node[circle,fill=black,scale=0.5,label={90:$v_{3,3}$}] at (210:6cm) (p3) {};
            \draw[->, thick] (center)--(n3);
        \end{tikzpicture}
        \caption{$\instar(3)$ with vertices labeled.}
        \label{fig:star-example}
    \end{figure}
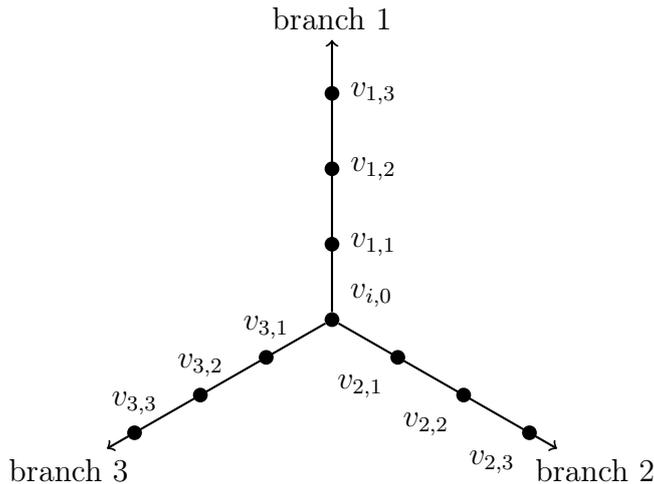Another fundamental result in chip-firing is called \textit{global confluence}:
    
    \begin{corollary}[Global confluence]\label{cor:global-confluence}
        \cite{klivansBook} Let $C$ be a configuration such that there is some finite sequence of legal firing moves to reach a stable configuration $C_s$. Then $C_s$ is the unique stable configuration reachable from $C$. We call this property global confluence.
        
        It follows from the proof of global confluence given in \cite{klivansBook} that although the order of legal firing moves is changeable via confluence, the stabilizing sequence is unique as a multi-set in the following sense:
        \begin{itemize}
            \item The length of any stabilizing sequence is the same, and
            
            \item The number of times each site fires in any stabilizing sequence is the same.
        \end{itemize}
    \end{corollary}
    
    We will use confluence to prove Theorem \ref{thm:stable-shape}, which determines the final number of chips on each vertex after stabilizing from a special starting configuration.
    
    \begin{proposition}
        For a configuration with $N=km+r$ $(0\leq r<k)$ unlabeled chips sitting on the center of $\instar(k)$, all reachable stable configurations have:
        \[C(v_{i,j})=\begin{cases}
            r &\text{if }j=0 \\
            1 &\text{if }j\in[m] \\
            0 &\text{otherwise}
        \end{cases}\]
        so $r$ chips on the center, $1$ chip on each of the innermost $m$ vertices of each branch, and no chips elsewhere.
    \end{proposition}
    
    In particular, we are interested in the special case where $r=0$. Let $[\Delta^{k,m}]$ denote the configuration where $N=km$ unlabeled chips all sit on the center of $\instar(k)$.
    
    \begin{theorem}\label{thm:stable-shape}
        If we start from $[\Delta^{k,m}]$, the unique reachable stable configuration is
        \[C(v_{i,j})=\begin{cases}
            1 &\text{if }j\in[m] \\
            0 &\text{otherwise}
        \end{cases}\]
        so the center is empty, the next $m$ vertices of every branch are filled, and the rest of the vertices are empty.
    
        Additionally, every stabilization sequence starting from $[\Delta^{k,m}]$ takes 
        \[\frac{m(m+1)}{2}+k\cdot\frac{(m-1)m(m+1)}{6}\]
        fires, with each vertex $v_{i,j}$ firing $\frac{(m-j)(m-j+1)}{2}$ times for $j\in[0,m-1]$ and $0$ times for all other vertices.
    \end{theorem}
    
    \begin{proof}
        Note that due to the global confluence property of unlabeled chip-firing given in Corollary \ref{cor:global-confluence}, it suffices to show that there exists one stabilization sequence with the desired properties.
    
        Consider the case where $k=1$. We will induct on $m$. As a base case, let $m=1$. We begin from $[\Delta^{1,1}]$ where $C(v_{1,0})=1$ and $C(v_{1,j})=0$ for $j\neq 0$. The center $v_{1,0}$ is ready to fire and fires once, sending its one chip to $v_{1,1}$. This gives $C(v_{1,1})=1$ and $C(v_{1,j})=0$ for $j\neq 1$ as desired. Additionally, stabilization took 1 fire total, with $\frac{(1-0)(1-0+1)}{2}=1$ fire at $v_{1,0}$ as desired.
    
        Now induct on $m$. Suppose that for all $m\leq m_0$, $[\Delta^{k,m}]$ has the desired stable configuration and each vertex fires $\frac{(m-j)(m-j+1)}{2}$ times. We will show that stabilization sequences of $[\Delta^{1,m_0+1}]$ have the desired properties as well, namely that the stable configuration will have $C(v_{1,j})=1$ for $j\in[m_0+1]$ and $C(v_{1,j})=0$ for all other vertices and that every vertex $v_{i,j}$ for $j\in[0,m_0]$ fires $\frac{(m_0-j+1)(m_0-j+2)}{2}$ times.
    
        Starting from $[\Delta^{1,m_0+1}]$, fire the same vertices in the same order as they would be fired in a stabilization sequence of $[\Delta^{1,m_0}]$. This leaves us with a configuration $\hat{C}$ where $\hat{C}(v_{i,j})=1$ for all $j\in[0,m_0]$. The only firable vertex is $v_{1,0}$. Fire $v_{1,0}$ to leave no chips on the center, two on $v_{1,1}$, and the remaining vertices unchanged. Now fire $v_{1,1}$ to leave one chip on $v_{1,0}$, none on $v_{1,1}$, and two on $v_{1,2}$. Continue firing the firable vertices from the center outwards until there is one chip on each of the vertices $\{v_{1,j}\}_{j=0}^{m_0-1}$ and $v_{1,m_0+1}$ and no chips on $v_{i,m_0}$. Then fire the vertices from the center outwards again until all vertices $\{v_{1,j}\}_{j=0}^{m_0+1}$ have one chip except $v_{i,m_0-1}$, which has none. Continue in this fashion until stabilized. This will result in the stable configuration $C$ with $C(v_{1,j})=1$ for $j\in[m_0+1]$ and $C(v_{1,j})=0$ for all other vertices, as desired. Additionally, during this process, each vertex in $\{v_{1,j}\}_{j-0}^{m_0}$ fires another $m_0-j+1$ times, bringing the total number of fires for a chip $v_{1,j}$ ($j\in[0,m_0]$) to
        \[\frac{(m_0-j)(m_0-j+1)}{2}+m_0-j+1= \frac{(m_0-j+1)(m_0-j+2)}{2}.\]
        Thus for all $m\geq 1$, $[\Delta^{1,m}]$ stabilizes to a configuration $C$ such that $C(v_{i,j})=1$ for $i\in[1]=\{1\}$, $j\in[m]$ and $C(v_{i,j})=0$ for $i=1$, $j\not\in[m]$, and additionally each vertex $v_{i,j}$ fires $\frac{(m-j)(m-j+1)}{2}$ times for $i=1$, $j\in[0,m-1]$ as desired. Finally, this means that the total number of fires in a stabilization sequence is the sum of the first $m$ triangular numbers which is given by the tetrahedral numbers:
        \[\sum_{j=0}^{m-1} \frac{(m-j)(m-j+1)}{2} = \frac{m(m+1)(m+2)}{6}.\]
    
        Suppose now that $k>1$. All of the same logic holds, except that $v_{i,0}$ must have $k$ chips to fire. So long as the firings are done in the manner described above, first firing the center, then firing the center followed by the level 1 vertices, then firing the center followed by the level 1 vertices followed by the level 2 vertices, etc., all of the logic is the same. The only change is that for $m>1$ the total number of fires becomes
        \[\frac{m(m+1)}{2} + \sum_{j=1}^{m-1} k\cdot\frac{(m-j)(m-j+1)}{2} = \frac{m(m+1)}{2}+k\cdot\frac{(m-1)m(m+1)}{6}\]
        (of which the formula given above for $k=1$ is a special case).
    \end{proof}
	
    \justify
    \section[Labeled Chip-Firing on the Infinite k-Star]{Labeled Chip-Firing on the Infinite $k$-Star}\label{sec:labeled-star}
    \subsection[Labeled Chip-Firing on the Infinite k-Star]{Labeled Chip-Firing on the Infinite $k$-Star}\label{subsec:labeled-star}
    
    Labeled chip-firing was introduced in \cite{Hopkins_2017} when the underlying graph is $\Z$. We adapt this process for the context of $\instar(k)$: Cyclically order the branches $1$ through $k$. Each chip is labeled with a distinct number in $[N]$, and we will denote a chip with label $c$ by $(c)$.  When firing a vertex $v_{i,j}$, we choose $\deg(v_{i,j})$ chips on the vertex and take note of their labels. We have the following two cases:
    \begin{itemize}
        \item When firing $v_{i,0}$, we choose $k$ chips at the center with labels $n_1<\cdots < n_k$. Chip $(n_i)$ is sent to $v_{i,1}$.
    
        \item When firing $v_{i,j}$ for $j\neq 0$, we choose $2$ chips on the vertex with labels $a<b$. Chip $(a)$ is sent to $v_{i,j-1}$ and chip $(b)$ is sent to $v_{i,j+1}$.
    \end{itemize}

    \justify
    \subsection{Main Theorem}\label{subsec:main-theorem}
    In this section, we consider the case of labeled chip-firing starting with chips $1$ through $km=N$ on the center vertex of $\instar(k)$. Denote this starting configuration ${\Delta^{k,m}}$ (so $[\Delta^{k,m}]$ is its underlying unlabeled configuration). We use $\mathcal{C}$ to denote a labeled chip configuration and $\mathcal{C}(v)$ to be the set of labeled chips on a vertex $v$. Thus if $\mathcal{C}$ has underlying unlabeled configuration $C$, we have that $|\mathcal{C}(v)|=C(v)$.

    Our main theorem is as follows:
    
    \begin{theorem}\label{thm:sorting}
        All reachable stable configurations starting from $\Delta^{k,m}$ are sorted in the following way: the chips on the branches are sorted in increasing order reading from the center outwards. In other words, for all $i\in[k]$ and all $j\in[m-1]$, the chip ending on $v_{i,j}$ is less than the chip ending on $v_{i,j+1}$.
    \end{theorem}
    
    \noindent The proof of Theorem \ref{thm:sorting} will follow after some crucial lemmas. In this section, we will:
    \begin{enumerate}
        \item Prove that there are special fires, which we call \textit{endgame fires}, that occur near the end of all stabilization sequences starting from $[\Delta^{k,m}]$. A partial ordering can be imposed on these endgame fires according to which fires must always occur before other fires in a stabilization sequence. (Lemma \ref{lem:poset})
        
        \item Prove that for every endgame fire at the center, the chip sent to branch $i$ is less than or equal to the chip that was sent to branch $i$ by the previous endgame fire at the center. (Lemma \ref{lem:mixing})
        
        \item Use Lemma \ref{lem:mixing} together to prove Theorem \ref{thm:sorting}.
    \end{enumerate}
    
    We begin by adapting the poset of vertex fires developed by Klivans and Liscio in \cite{klivans2020confluencelabeledchipfiring}. Starting from a fixed configuration, in our case $[\Delta^{k,m}]$ in particular, let $v_{i,j}^f$ denote the $f^{th}$-last fire of vertex $v_{i,j}$ in a stabilization sequence (with $f=0$ being the last fire of a vertex, $f=1$ being the next to last, etc.) Note that this is different than the convention established in \cite{klivans2020confluencelabeledchipfiring}, where $f=1$ would correspond to the final firing of a site. Say $v_{i,j}^f>v_{i',j'}^{f'}$ if $v_{i,j}^f$ must occur before $v_{i',j'}^{f'}$ in any stabilization sequence. This gives a poset structure on the fires. We focus in particular on a specific set of fires that happen at the end of each stabilization sequence, which we refer to as \textit{endgame fires} (in the style of Musiker and Nguyen's ``endgame moves'' \cite{musiker2023labeledchipfiringbinarytrees}). For any vertex $v_{i,j}$ that fires, its endgame fires are its final $m-j$ fires; that is, the endgame fires for $[\Delta^{k,m}]$ are $\{v_{i,j}^f:i\in[k],j\in[0,m-1],f\in[0,m-j-1]\}$.
    
    \begin{lemma}\label{lem:poset}
        The poset of endgame fires of the stabilization sequences for $[\Delta^{k,m}]$ has a particular grid-like structure. Let $v_{i,j}^f$ be an endgame fire. Then the following conditions must hold:
        
        \begin{enumerate}
            \item[(1a)] When $j\neq 0$, $v_{i,j}^f \leq v_{i,j-1}^{f+1}$ and $v_{i,j}^f \leq v_{i,j+1}^f$ (if $v_{i,j-1}^{f+1}$ and $v_{i,j+1}^f$ are endgame moves, respectively).
            
            \item[(1b)] $v_{i,0}^f \leq v_{i,1}^f$ for all $i\in[k]$ (if $f<m-1$).
            
            \item[(2)] When the firing move $v_{i,j}^f$ occurs, there are exactly $\deg(v_{i,j})$ chips present at the site that is firing.
        \end{enumerate}
    
        This results in a Hasse diagram made of $k$ ``half-diamond grids'', one for each branch of the $\instar(k)$ graph, like the example shown in Figure \ref{fig:poset}. We call this a \textup{book poset}, and each half-diamond grid a \textup{page} of the book. These half-diamond grids are disjoint except for at the elements $\{v_{i,0}^f\}_{f=0}^{m-1}$, which are shared among all half-diamond grids of the poset.
    \end{lemma}
    
    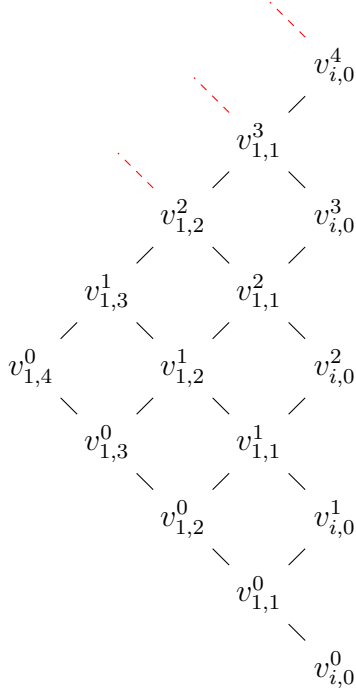
\begin{figure}
    \begin{center}
        \begin{tikzpicture}
    	      \node (0) at (0, 4) {$v_{i,0}^4$};
    	      \node[red] (19) at (-1, 5) {};
    		\node (1) at (0, 2) {$v_{i,0}^3$};
    		\node (3) at (0, 0) {$v_{i,0}^2$};
    		\node (5) at (0, -2) {$v_{i,0}^1$};
    		\node (6) at (0, -4) {$v_{i,0}^0$};
    		\node (7) at (-1, -3) {$v_{1,1}^0$};
    		\node (8) at (-1, -1) {$v_{1,1}^1$};
    		\node (9) at (-1, 1) {$v_{1,1}^2$};
    		\node (10) at (-1, 3) {$v_{1,1}^3$};
    		\node[red] (17) at (-2, 4) {};
    		\node (11) at (-2, 2) {$v_{1,2}^2$};
    		\node[red] (18) at (-3, 3) {};
    		\node (12) at (-2, 0) {$v_{1,2}^1$};
    		\node (13) at (-2, -2) {$v_{1,2}^0$};
    		\node (14) at (-3, -1) {$v_{1,3}^0$};
    		\node (15) at (-3, 1) {$v_{1,3}^1$};
    		\node (16) at (-4, 0) {$v_{1,4}^0$};
            \draw (0) to (10);
            \draw (10) to (11);
            \draw (11) to (15);
            \draw (15) to (16);
            \draw (10) to (1);
            \draw (11) to (9);
            \draw (15) to (12);
            \draw (16) to (14);
            \draw (1) to (9);
            \draw (9) to (12);
            \draw (12) to (14);
            \draw (9) to (3);
            \draw (12) to (8);
            \draw (14) to (13);
            \draw (3) to (8);
            \draw (8) to (13);
            \draw (8) to (5);
            \draw (13) to (7);
            \draw (5) to (7);
            \draw (7) to (6);
            \draw[red,dashed] (10) to (17);
            \draw[red,dashed] (11) to (18);
            \draw[red,dashed] (0) to (19);
        \end{tikzpicture}
        \caption{\justifying\label{fig:poset} An example of one page corresponding to the first branch of the book poset of $[\Delta^{k,5}]$ for some $k$. The red dotted lines indicate relations between elements of the endgame fires of the book poset and non-endgame fires, which we do not discuss in detail in our paper.}
    \end{center}
    \end{figure}
    
    The proof that our chip-firing set-up has this poset structure follows immediately from the proof of Lemma 2.6 in \cite{klivans2020confluencelabeledchipfiring}. Indeed, one can follow the proof of this lemma to show that each branch has this grid structure, and can then combine the ``pages'' together to get our full poset.
    
    \begin{proof}
        We note that following Klivans and Liscio's proof, we can use the same logic and Theorem \ref{thm:stable-shape} to show that the last fire from the center vertex has to occur after the last fires of the level 1 vertices on each of the branches connected to it; further, that it must fire with exactly $k$ chips (paragraph 1 in the proof of Lemma 2.6 in \cite{klivans2020confluencelabeledchipfiring}). This gives us the bottom ``tip'' of our poset.
        
        Now we may induct on the ``edges'' of our book poset. That is, if fire $v_{i,j}^0$ on any branch $i$ must take place before $v_{i,j-1}^0$, then we want to show that $v_{i,j+1}^0$ and $v_{i,j-1}^1$ must take place before $v_{i,j}^0$. This can be seen applying the same logic as Klivans and Liscio's proof and Theorem \ref{thm:stable-shape} (paragraph 2 of Lemma 2.6 proof).
        
        Now we induct on rows of the poset to fill in the grid structure. For a fixed row of the poset, assume that $v_{i,j}^f$ must take place before $v_{i,j-1}^f$ and $v_{i,j+1}^{f-1}$, which both take place before $v_{i,j}^{f-1}$. Also assume that fires $v_{i,j+1}^{f-1}$, $v_{i,j-1}^f$, and $v_{i,j}^{f-1}$ all occur when the vertices have the minimum number of chips necessary to fire. We want to show that moves $v_{i,j+1}^f$ and $v_{i,j-1}^{f+1}$ must take place before $v_{i,j}^f$. This again follows from Klivans and Lisico's logic (paragraphs 3 and 4 of Lemma 2.6 proof) and Theorem \ref{thm:stable-shape}. The vertices where $j=0$ and $j\neq 0$ need to be considered separately, but both follow the same logic.
    \end{proof}
    
    \begin{lemma}\label{lem:mixing}
        Suppose that the endgame fire $v_{i,0}^f$ sends chip $(c'_i)$ to branch $i$. Then the endgame fire $v_{i,0}^{f-1}$ (if it exists) sends a chip smaller than $(c'_i)$ to branch $i$. In other words, each endgame center fire sends a smaller chip to a given branch than the endgame center fire before.
    \end{lemma}
    
    \begin{proof}
        When the second center fire occurs there are three possibilities: branch $i$ receives the same chip it sent to the center, branch $i$ receives a chip from a branch $\ell<i$, or branch $i$ receives a chip from a branch $\ell>i$. Case $1$ obviously fulfills the desired property. In case $2$, we know that the largest chip we can receive from $\ell$ is bounded above by $c_{\ell}'<c_i'$, thus we must inherit a smaller chip. In case $3$, we consider that cycle that contains $i$ and $\ell$, and note that this implies that a branch larger than $i$ must have inherited a chip from a branch smaller than or equal to $i$, say chip $(c_a)$ from branch $a$. Then $c_a>c_i$, and furthermore $c_i<c_a\leq c'_a\leq c_i'$ since $a\leq i$. Thus the chip that branch $i$ received must have label smaller than $c_i$.
    \end{proof}
    
    Finally, we present a proof of Theorem \ref{thm:sorting}:
    
    \begin{proof}
        We begin by noting that Lemma $\ref{lem:mixing}$ guarantees that after fire $v^{m-1}_{i,0}$ occurs, the largest chip that will be on each branch is already there. Call this chip $(N_i)$. From this point, we note that $(N_i)$ will be at most $m-1$ positions away from $v_{i,m}$. Now, when fire $v^{m-2}_{i,1}$ occurs, $(N_i)$ will at be at most $m-2$ positions away from $v_{i,m}$. We continue with this logic and note that fire $v^{0}_{i,m-1}$ guarantees that $(N_i)$ will end up at position $v_{i,m}$. 
    
        By the same logic we know that once fire $v^{m-2}_{i,0}$ occurs, the second largest chip that will end on branch $i$ will already be on branch $i$. Now, we note that after $v^{m-2}_{i,0}$, this chip can at most be $m-2$ spots away from $v_{i,m-1}$. Furthermore, when $v^{m-3}_{i,1}$ occurs, it guarantees that this second largest chip will be at most $m-3$ positions away from $v_{i,m-1}$ (at this point in the firing, due to our poset structure, we know that this chip will never run into the largest chip). Thus by the same logic as above, we see the second largest chip will end up at position $v_{i,m-1}$.
    
        Now, by induction we see that the chips will by ordered from smallest to largest on each branch as desired.
    \end{proof}
    
    In addition to Theorem \ref{thm:sorting} which showed that each individual branch of $instar(k)$ is sorted, our next result shows a sorting property also holds for the circles of chips closest and furthest from the center.
    
    \begin{proposition}\label{prop:level1andm}
        Additionally, level 1 and level $m$ vertices are also sorted, in the sense that the chip ending on $v_{i,1}$ is smaller than the chip ending on $v_{i+1,1}$, and likewise the chip ending on $v_{i,m}$ is smaller than the chip ending on $v_{i+1,m}$, for all $i\in[k-1]$.
    \end{proposition}
    
    \begin{proof}
        The sorting of the level $1$ vertices follows immediately from the fact that the last fire must be $v_{i,0}^0$.
    
        We now show that the level $m$ vertices sort as well: Consider the largest chip $(c)$ on branch $i$. Then we consider the fire that sent $(c)$ from the center to branch $i$. During this fire a larger chip with label $d>c$ must have been sent to branch $i+1$, and the only way for $(d)$ to return to the center is if a chip larger than $(d)$ fires into the branch. Thus the largest chip on branch $i$ must be smaller than the largest chip on branch $i+1$.
    \end{proof}

    \justify
    \section{Categorization of Reachable Stable Configurations}\label{sec:reachable-characterization}
    \subsection[Shape (k x 2) SYT and Catalan Numbers]{Shape $k\times 2$ SYT and Catalan Numbers}\label{subsec:Catalan-SYTs}

    Some reachable stable configurations are particularly nice. For example, when $m=2$, the reachable stable configurations are counted by the Catalan numbers, which we show using standard Young tableaux. A \textit{standard Young tableau} is a Young diagram of size $\lambda=(\lambda_1,\lambda_2,\ldots,\lambda_k)\in\Z^k$, where $\sum_{i=1}^k\lambda_i=N$, that has been filled such that
    \begin{enumerate}
        \item each cell is filled with a unique number in $[N]$,
        
        \item the entries of each row are strictly increasing as read from left to right, and
        
        \item the entries of each column are strictly increasing as read from top to bottom.
    \end{enumerate}
    It is well known (see, for instance, Item (ww) in \cite{stanleycatalanexcerpt}) that standard Young tableaux of shape $(2,\ldots,2)\in\Z^k$ are counted by the Catalan number $\frac{1}{1+k}\binom{2k}{k}$. Henceforth, we will use $(2)^k:=(2,\ldots,2)\in\Z^k$.
    
    \begin{theorem}\label{thm:Catalan-bijection}
        Let $\T^{k,m}$ denote the set of all reachable stable configurations of $\Delta^{k,m}$, and let $T^{k,m}:=|\T^{k,m}|$.
        
        We let $m=2$. The configurations of $\T^{k,2}$ are in bijection with standard Young tableaux of shape $(2)^k$ filled with the integers in $[2k]$, and are thus counted by the Catalan numbers.
    
        The bijection $\varphi:\T^{k,2}\to\{Y:Y\text{ is a standard Young tableau of shape }(2)^k\}$ is as follows: given a configuration $\mathcal{C}\in\T^{k,2}$, we have
        \[\varphi(\mathcal{C}) = \ytableausetup{centertableaux, boxsize = 3.5em}
        \begin{ytableau}
            \mathcal{C}(v_{1,1}) & \mathcal{C}(v_{1,2}) \\
            \mathcal{C}(v_{2,1}) & \mathcal{C}(v_{2,2}) \\
            \vdots & \vdots \\
            \mathcal{C}(v_{k,1}) & \mathcal{C}(v_{k,2}) \\
        \end{ytableau}\]
        where each row of the tableau represents a branch of $\instar(k)$ and each column a level of vertices, and for a standard Young tableau
        \[Y=\ytableausetup{centertableaux, boxsize = 2em}
        \begin{ytableau}
            y_{1,1} & y_{1,2} \\
            y_{2,1} & y_{2,2} \\
            \vdots & \vdots \\
            y_{k,1} & y_{k,2} \\
        \end{ytableau}\]
        we have $\varphi^{-1}(Y)=\mathcal{C}$ where 
        \[\mathcal{C}(v_{i,j}) = \begin{cases}
            \{(y_{i,j})\} &\text{if }(i,j)\in[k]\times[2] \\
            \emptyset &\text{otherwise}
        \end{cases}.\]
    \end{theorem}
    
    \begin{proof}
        First, we show that every reachable stable configuration can be mapped to a standard Young tableau of shape $(2)^k$. Suppose we have some reachable stable configuration $\mathcal{C}$. We know by Theorem \ref{thm:stable-shape} that there is one chip on each vertex $v_{i,j}$ for $(i,j)\in[k]\times[m]$, so $\varphi(\mathcal{C})$ is a shape $(2)^k$ Young tableau. By Theorem \ref{thm:sorting}, we know that $\varphi(\mathcal{C})$ will be sorted in increasing order along its rows. Additionally, by Proposition \ref{prop:level1andm}, it must also be sorted along its first and last, and in fact only two, columns. Thus $\varphi(\mathcal{C})$ is a shape $(2)^k$ standard Young tableau.
    
        Next, we show that every standard Young tableau of shape $(2)^k$ maps by $\varphi^{-1}$ to a unique reachable stable configuration of $\Delta^{k,2}$. Starting from $\Delta^{k,2}$, fire chips $\{(y_{i,2})\}_{i=1}^k$ from the center. Each $(y_{i,2})$ is necessarily sent to $v_{i,1}$, the innermost vertex of the corresponding branch, because $Y$ must be sorted in increasing order along columns. Next, fire the center again. Now all vertices are empty except those in $\{v_{i,2}\}_{i=1}^k$, which have $\mathcal{C}(v_{i,j})=\{(y_{i,1}),(y_{i,2})\}$. Next, fire all of these vertices. The chips $(y_{i,1})$ will be sent back to the center while the $(y_{i,2})$ chips will be sent to $v_{i,2}$, since $Y$ must be sorted in increasing order along its rows. Finally, fire the center again to send each $(y_{i,1})$ chip to $v_{i,1}$. Thus
        \[\mathcal{C}(v_{i,j})=\begin{cases}
            \{(y_{i,j})\} &\text{if }(i,j)\in[k]\times[2] \\
            \emptyset &\text{otherwise}
        \end{cases}\]
        as desired. Thus each standard Young tableau of shape $(2)^k$ maps to a reachable stable configuration. It is clear that each of these configurations will be unique.
    \end{proof}

    \justify
    \subsection{Rectangular SYTs}\label{subsec:SYTs}

    A bijection between elements of $\T^{k,m}$ and standard Young tableaux of shape $(m)^k$ does not exist in this way for $m>2$.
    
    \begin{example}\label{ex:SYT-counterexample}
        Suppose $(k,m)=(3,3)$. Observe the sequence of fires shown in Figure \ref{fig:catalan-counterexample} starting from $\Delta^{3,3}$. The configuration shown would correspond to the tableau
        \[\ytableausetup{centertableaux, boxsize = 2em}
        \begin{ytableau}
            1 & \RED{4} & 7 \\
            2 & \RED{3} & 8 \\
            5 & 6 & 9 \\
        \end{ytableau}\]
        which is non-standard (note column 2).
    \end{example}
    
        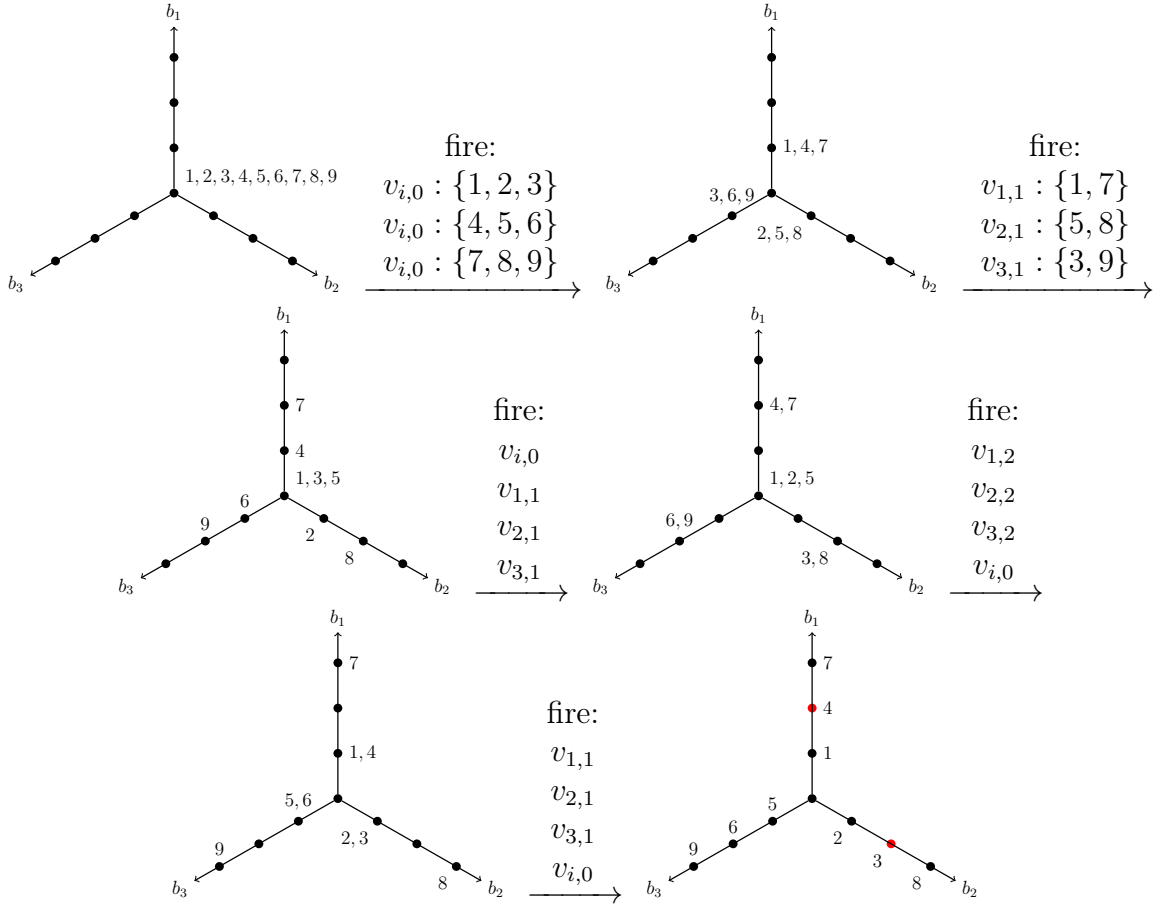
\begin{figure}[ht]
            \centering
            \scalebox{0.6}{
            \begin{tikzpicture}[scale=0.5]
                \node[circle,fill=black,scale=0.5,label={375:$1,2,3,4,5,6,7,8,9$}] at (360:0mm) (center) {};
                \node at (90:8cm) (n1) {$b_1$};
                \node[circle,fill=black,scale=0.5] at (90:2cm) (m1) {};
                \node[circle,fill=black,scale=0.5] at (90:4cm) (l1) {};
                \node[circle,fill=black,scale=0.5] at (90:6cm) (p1) {};
                \draw[->, thick] (center)--(n1);
                \node at (330:8cm) (n2) {$b_2$};
                \node[circle,fill=black,scale=0.5] at (330:2cm) (m2) {};
                \node[circle,fill=black,scale=0.5] at (330:4cm) (l2) {};
                \node[circle,fill=black,scale=0.5] at (330:6cm) (p2) {};
                \draw[->, thick] (center)--(n2);
                \node at (210:8cm) (n3) {$b_3$};
                \node[circle,fill=black,scale=0.5] at (210:2cm) (m3) {};
                \node[circle,fill=black,scale=0.5] at (210:4cm) (l3) {};
                \node[circle,fill=black,scale=0.5] at (210:6cm) (p3) {};
                \draw[->, thick] (center)--(n3);
            \end{tikzpicture}
            }
            $\xrightarrow{\begin{array}{c}
                \text{fire:} \\
                v_{i,0}:\{1,2,3\} \\
                v_{i,0}:\{4,5,6\} \\
                v_{i,0}:\{7,8,9\} \\
            \end{array}}$
            \scalebox{0.6}{
            \begin{tikzpicture}[scale=0.5]
                \node[circle,fill=black,scale=0.5] at (360:0mm) (center) {};
                \node at (90:8cm) (n1) {$b_1$};
                \node[circle,fill=black,scale=0.5,label={360:$1,4,7$}] at (90:2cm) (m1) {};
                \node[circle,fill=black,scale=0.5] at (90:4cm) (l1) {};
                \node[circle,fill=black,scale=0.5] at (90:6cm) (p1) {};
                \draw[->, thick] (center)--(n1);
                \node at (330:8cm) (n2) {$b_2$};
                \node[circle,fill=black,scale=0.5,label={-120:$2,5,8$}] at (330:2cm) (m2) {};
                \node[circle,fill=black,scale=0.5] at (330:4cm) (l2) {};
                \node[circle,fill=black,scale=0.5] at (330:6cm) (p2) {};
                \draw[->, thick] (center)--(n2);
                \node at (210:8cm) (n3) {$b_3$};
                \node[circle,fill=black,scale=0.5,label={90:$3,6,9$}] at (210:2cm) (m3) {};
                \node[circle,fill=black,scale=0.5] at (210:4cm) (l3) {};
                \node[circle,fill=black,scale=0.5] at (210:6cm) (p3) {};
                \draw[->, thick] (center)--(n3);
            \end{tikzpicture}
            }
            $\xrightarrow{\begin{array}{c}
                \text{fire:} \\
                v_{1,1}:\{1,7\} \\
                v_{2,1}:\{5,8\} \\
                v_{3,1}:\{3,9\} \\
            \end{array}}$
            \scalebox{0.6}{
            \begin{tikzpicture}[scale=0.5]
                \node[circle,fill=black,scale=0.5,label={375:$1,3,5$}] at (360:0mm) (center) {};
                
                \node at (90:8cm) (n1) {$b_1$};
                \node[circle,fill=black,scale=0.5,label={360:$4$}] at (90:2cm) (m1) {};
                \node[circle,fill=black,scale=0.5,label={360:$7$}] at (90:4cm) (l1) {};
                \node[circle,fill=black,scale=0.5] at (90:6cm) (p1) {};
                \draw[->, thick] (center)--(n1);
                \node at (330:8cm) (n2) {$b_2$};
                \node[circle,fill=black,scale=0.5,label={-120:$2$}] at (330:2cm) (m2) {};
                \node[circle,fill=black,scale=0.5,label={-120:$8$}] at (330:4cm) (l2) {};
                \node[circle,fill=black,scale=0.5] at (330:6cm) (p2) {};
                \draw[->, thick] (center)--(n2);
                \node at (210:8cm) (n3) {$b_3$};
                \node[circle,fill=black,scale=0.5,label={90:$6$}] at (210:2cm) (m3) {};
                \node[circle,fill=black,scale=0.5,label={90:$9$}] at (210:4cm) (l3) {};
                \node[circle,fill=black,scale=0.5] at (210:6cm) (p3) {};
                \draw[->, thick] (center)--(n3);
            \end{tikzpicture}
            }
            $\xrightarrow{\begin{array}{c}
                \text{fire:} \\
                v_{i,0} \\
                v_{1,1} \\
                v_{2,1} \\
                v_{3,1} \\
            \end{array}}$
            \scalebox{0.6}{
            \begin{tikzpicture}[scale=0.5]
                \node[circle,fill=black,scale=0.5,label={375:$1,2,5$}] at (360:0mm) (center) {};
                \node at (90:8cm) (n1) {$b_1$};
                \node[circle,fill=black,scale=0.5] at (90:2cm) (m1) {};
                \node[circle,fill=black,scale=0.5,label={360:$4,7$}] at (90:4cm) (l1) {};
                \node[circle,fill=black,scale=0.5] at (90:6cm) (p1) {};
                \draw[->, thick] (center)--(n1);
                \node at (330:8cm) (n2) {$b_2$};
                \node[circle,fill=black,scale=0.5] at (330:2cm) (m2) {};
                \node[circle,fill=black,scale=0.5,label={-120:$3,8$}] at (330:4cm) (l2) {};
                \node[circle,fill=black,scale=0.5] at (330:6cm) (p2) {};
                \draw[->, thick] (center)--(n2);
                \node at (210:8cm) (n3) {$b_3$};
                \node[circle,fill=black,scale=0.5] at (210:2cm) (m3) {};
                \node[circle,fill=black,scale=0.5,label={90:$6,9$}] at (210:4cm) (l3) {};
                \node[circle,fill=black,scale=0.5] at (210:6cm) (p3) {};
                \draw[->, thick] (center)--(n3);
            \end{tikzpicture}
            }
            $\xrightarrow{\begin{array}{c}
                \text{fire:} \\
                v_{1,2} \\
                v_{2,2} \\
                v_{3,2} \\
                v_{i,0} \\
            \end{array}}$
            \scalebox{0.6}{
            \begin{tikzpicture}[scale=0.5]
                \node[circle,fill=black,scale=0.5] at (360:0mm) (center) {};
                \node at (90:8cm) (n1) {$b_1$};
                \node[circle,fill=black,scale=0.5,label={360:$1,4$}] at (90:2cm) (m1) {};
                \node[circle,fill=black,scale=0.5] at (90:4cm) (l1) {};
                \node[circle,fill=black,scale=0.5,label={360:$7$}] at (90:6cm) (p1) {};
                \draw[->, thick] (center)--(n1);
                \node at (330:8cm) (n2) {$b_2$};
                \node[circle,fill=black,scale=0.5,label={-120:$2,3$}] at (330:2cm) (m2) {};
                \node[circle,fill=black,scale=0.5] at (330:4cm) (l2) {};
                \node[circle,fill=black,scale=0.5,label={-120:$8$}] at (330:6cm) (p2) {};
                \draw[->, thick] (center)--(n2);
                \node at (210:8cm) (n3) {$b_3$};
                \node[circle,fill=black,scale=0.5,label={90:$5,6$}] at (210:2cm) (m3) {};
                \node[circle,fill=black,scale=0.5] at (210:4cm) (l3) {};
                \node[circle,fill=black,scale=0.5,label={90:$9$}] at (210:6cm) (p3) {};
                \draw[->, thick] (center)--(n3);
            \end{tikzpicture}
            }
            $\xrightarrow{\begin{array}{c}
                \text{fire:} \\
                v_{1,1} \\
                v_{2,1} \\
                v_{3,1} \\
                v_{i,0} \\
            \end{array}}$
            \scalebox{0.6}{
            \begin{tikzpicture}[scale=0.5]
                \node[circle,fill=black,scale=0.5] at (360:0mm) (center) {};
                \node at (90:8cm) (n1) {$b_1$};
                \node[circle,fill=black,scale=0.5,label={360:$1$}] at (90:2cm) (m1) {};
                \node[circle,fill=red,scale=0.5,label={360:$4$}] at (90:4cm) (l1) {};
                \node[circle,fill=black,scale=0.5,label={360:$7$}] at (90:6cm) (p1) {};
                \draw[->, thick] (center)--(n1);
                \node at (330:8cm) (n2) {$b_2$};
                \node[circle,fill=black,scale=0.5,label={-120:$2$}] at (330:2cm) (m2) {};
                \node[circle,fill=red,scale=0.5,label={-120:$3$}] at (330:4cm) (l2) {};
                \node[circle,fill=black,scale=0.5,label={-120:$8$}] at (330:6cm) (p2) {};
                \draw[->, thick] (center)--(n2);
                \node at (210:8cm) (n3) {$b_3$};
                \node[circle,fill=black,scale=0.5,label={90:$5$}] at (210:2cm) (m3) {};
                \node[circle,fill=black,scale=0.5,label={90:$6$}] at (210:4cm) (l3) {};
                \node[circle,fill=black,scale=0.5,label={90:$9$}] at (210:6cm) (p3) {};
                \draw[->, thick] (center)--(n3);
            \end{tikzpicture}
            }
            
            \caption{\justifying Counterexample to a generalized $\T^{k,m}\leftrightarrow \{${standard Young tableaux of shape }$(m)^k\}$ bijection.}
            \label{fig:catalan-counterexample}
        \end{figure}
    
    Still, we have that all reachable stable configurations correspond to row-increasing tableaux (though not every row-increasing tableau correspond to reachable configurations), and in fact we know every standard Young tableaux of shape $(m)^k$ still corresponds to an element of $\T^{k,m}$:
    
    \begin{theorem}\label{SYT-to-stable-only}
        Every element of $\{Y:Y\text{ is a standard Young tableau of shape }(m)^k\}$ corresponds to a unique element $\mathcal{C}$ of $\T^{k,m}$, where $\mathcal{C}(v_{i,j})=\{(y_{i,j})\}$ for $(i,j)\in[k]\times[m]$.
    \end{theorem}
    
    \begin{proof}
        Consider the function $\psi:\{Y:Y\text{ is an SYT of shape }(m)^k\}\to\T^{k,m}$ such that
        \[\psi\left(\hspace{0.25cm}\begin{ytableau}
            y_{1,1} & y_{1,2} & \cdots & y_{1,m} \\
            y_{2,1} & y_{2,2} & \cdots & y_{2,m} \\
            \vdots  & \vdots  & \ddots & \vdots \\
            y_{k,1} & y_{k,2} & \cdots & y_{k,m} \\
        \end{ytableau}\hspace{0.25cm}\right) = \mathcal{C}\]
        with
        \[\mathcal{C}(v_{i,j})=\begin{cases}
            \{(y_{i,j})\} &\text{if }(i,j)\in[k]\times[m] \\
            \emptyset &\text{otherwise}.
        \end{cases}\]
    
        It is clear that by construction, any two distinct SYT will correspond to two distinct configurations on $\instar(k)$, so $\psi$ is injective. We want to show that these are actually reachable stable configurations.
    
        Let $Y$ be an $(m)^k$ SYT with entries indexed as shown above. First, from $\Delta^{k,m}$, fire the center with chips $\{(y_{i,m})\}_{i=1}^k$. We know the chip $(y_{i,m})$ goes to branch $i$ since $Y$ is a standard Young tableau and therefore entries increase down columns. Then fire the center again with chips $\{(y_{i,m-1})\}_{i=1}^k$. Now the level 1 vertices are ready to fire. Fire each level 1 vertex so that $(y_{i,m})$ goes to $v_{i,2}$ and $(y_{i,m-1})$ goes to $v_{i,0}$ for each $i\in[k]$. We know the $(y_{i,m})$ chips move away from the center since $Y$ is a standard Young tableau with entries increasing along rows. Now fire the center again with chips $\{(y_{i,m-1})\}_{i=1}^k$, and again with chips $\{(y_{i,m-2})\}_{i=1}^k$. Now the level 1 vertices are ready to fire. Fire them so that chips $(y_{i,m})$ and $(y_{i,m-1})$ are on $v_{i,2}$ and $(y_{i,m-2})$ goes back to the center for each $i\in[k]$. Now fire the level 2 vertices so that $(y_{i,m})$ lands on $v_{i,3}$ and $(y_{i,m-1})$ lands on $v_{i,2}$ for each $i\in[k]$. Continue in this manner until the configuration has been stabilized. The resulting configuration will be $\psi(Y)$.
    \end{proof}
    
    A natural question to ask is if we can guarantee that a final configuration will map to an SYT by imposing some constraints on the firing. In other words, we are interested in specific firing techniques that always yield these SYTs. For this purpose we introduce what we call ``volatility-minimizing'' firing and show that SYTs of shape $(m)^k$ are in bijection with the set of final configurations that occur under this firing strategy. 
    
    \begin{definition}
        We say a firing sequence is ``volatility-minimizing'' if at all steps in the chip-firing process, vertices can only be chosen to fire if they keep the number of ready-to-fire vertices to the minimum possible out of all firing options. If there are multiple vertices that satisfy this property, choose one that is the furthest from the center.
    \end{definition}
    
    It is not difficult to convince oneself that such a firing process will proceed as follows for the star graph: fire the center two times, fire each of the branches once at level $1$. Fire the center two more times, then for each branch, fire the level 1 vertex, then the level 2 vertex, and so on. It turns out that the image of $\psi$ is precisely the set of final configurations reached with volatility-minimizing firing sequences. Thus $T(k,m)$ from Entry A060854 of the OEIS (\cite{oeis}) counts the number of such configurations. When $m=2$, this specializes to the Catalan numbers. 
    
    In order to prove the final result we require the following lemma:
    
    \begin{lemma}\label{lem:Matrix-Sorting}(See, for instance, ``Rows and Columns" on page 23 of \cite{winkler})
        Given a matrix such that every column is entry-wise increasing one can create a matrix that is also row-increasing in the following way: take each row in the matrix and sort the entries of that row.
    \end{lemma}
    
    \begin{theorem}The set of final configurations resulting from ``volatility-minimizing'' firing sequences of  $\Delta^{k,m}$ is in bijection with the set of SYT of shape $(m)^k$.
    \end{theorem}
    
    \begin{proof} For this proof we again use the function $\psi$ to translate between SYT and chip configurations. Let us analyze the tableau that forms as we fire from $\Delta^{k,m}$. The first fire fills in the first column of an empty $k\times m$ matrix where the $i^{th}$ row's entry represents the chip on the $i^{th}$ ``smallest'' branch. Now, after the second fire we have two entries in the same spot for each position in the first column of our matrix. However, now we must fire each branch one time. The result is a new matrix with an empty first column and filled second column. Indeed, the fires sent half of the chips back to the center, and the other half to the next spot on each branch (column $2$). Note that Lemma \ref{lem:Matrix-Sorting} guarantees that column $2$ is still ordered.
    
    Now we note that under our firing strategy, at most $2$ chips will be on any given branch vertex at a time. Thus for a fixed $j$, the sequential firing of the vertices in $\{v_{1,j},\ldots, v_{k,j}\}$ can be viewed as a row sorting operation in the same way as in the previous paragraph. Thus at each stage in the firing process in which no vertices other than the origin have more than one chip, the columns of the matrix will be sorted.
    
    Now we note what happens after we enter the poset of endgame fires at $v_{i,0}^{m-1}$. From this point onward no vertex will have a number of chips greater than its degree. Thus any fires that occur on the center vertex from this point onward will come from firing the first vertex on each branch. However, by Lemma \ref{lem:Matrix-Sorting} we know that the collection of these fires will sort the chips into two increasing columns, so the chips that enter the center will never mix, but instead return to the branches from which they came. Now, we utilize the proof of Theorem \ref{thm:sorting} and Lemma \ref{lem:Matrix-Sorting} again. We have $k$ columns that are sorted and we know that chip-firing from this point onward will not ``mix'' chips between branches. That is to say, all chip firing does from this point forward is permute entries in rows. Thus we have a column sorted matrix that will become row sorted via Theorem \ref{thm:sorting} and does so by only permuting entries that lie in the same row. Thus by Lemma \ref{lem:Matrix-Sorting}, the resulting matrix must be an SYT.
    
    As for surjectivity, given any SYT, we can fire to it on the following way: Fire the largest column first, followed by the second largest column. Then fire the branch vertices. This results in the largest column being pushed out and the second column going back into the center. Refire the second column followed by the third largest column. Continuing this way it is easy to see that we obtain the desired SYT.
    \end{proof}

    \justify
    \section{Future Questions}
    
    A natural follow-up is to describe and enumerate \textit{all} stable configurations of $\Delta^{k,m}$ when $m\neq2$. This is easily done when either $k=1$ or $m=1$. Theorem \ref{thm:sorting} guarantees that there is only one reachable stable configuration in the $k=1$ case, while Proposition \ref{prop:level1andm} guarantees the same for the $m=1$ case. However, when $k>1$ and $m>2$, this task is more difficult.
    
    We know that all reachable stable configurations can be mapped to tableaux of shape $(m)^k$ (Theorem \ref{thm:stable-shape}) which are row-increasing and have their first and last columns increasing (Theorem \ref{thm:sorting} and Proposition \ref{prop:level1andm}), but this is not enough to categorize all stable configurations. For example, there are 20 Young tableaux of shape $(2)^4$ that meet these criteria, but only 16 stable configurations reachable from $\Delta^{2,4}$. The tableaux that do not correspond to stable configurations reachable from $\Delta^{2,4}$ are shown in Figure \ref{fig:fourtableaux}. We have yet to find a bijection between general stable configurations and an existing combinatorial object.
    \begin{figure}[ht]
        \centering
        \[\begin{array}{cccc}
            \begin{ytableau}
                1 & 2 & \BLUE{6} & 7 \\
                3 & 4 & \BLUE{5} & 8
            \end{ytableau} & \begin{ytableau}
                1 & 3 & \BLUE{6} & 7 \\
                2 & 4 & \BLUE{5} & 8
            \end{ytableau} & \begin{ytableau}
                1 & \BLUE{4} & \BLUE{6} & 7 \\
                2 & \BLUE{3} & \BLUE{5} & 8
            \end{ytableau} & \begin{ytableau}
                1 & \BLUE{5} & \BLUE{6} & 7 \\
                2 & \BLUE{3} & \BLUE{4} & 8
            \end{ytableau}
        \end{array}\]
        \caption{\justifying Tableaux of shape $(2)^4$ which are row-sorted and have first and last columns sorted but do not correspond to stable configurations reachable from $\Delta^{2,4}$. Non-increasing columns are highlighted in blue.}
        \label{fig:fourtableaux}
    \end{figure}
    
    Another question that came up is one of probability: What is the probability that when randomizing the firing process one reaches a stable configuration corresponding to an SYT? We have the following conjecture:
    
    \begin{conjecture}
        Suppose we play the labeled chip-firing game at random. That is, each firable vertex has the same probability of firing at a given stage, and once a vertex is chosen, each chip on the vertex is equally likely to fire. In this situation, the most likely stable configuration is the one where chips $(1)$ through $(m)$ land on branch 1, chips $(m+1)$ through $(2m)$ on branch 2, \ldots, and chips $\left((k-1)m+1\right)$ through $(km)$ on branch $k$. We refer to this as the ``totally sorted'' configuration for $\Delta^{k,m}$. Further, SYT configurations are more likely to occur than non-SYT configurations.
    \end{conjecture}
    
    This is based off of data collected from a program written in SageMath 10.4 using Python 3.11.11. The program takes in a pair $(k,m)$ and determines every stable configuration reachable from $\Delta^{k,m}$ by running through every possible stabilization sequence. Additionally, the program records the number of different stabilization sequences ending in each stable configuration. This program has given us the data shown in Figure \ref{fig:probabilities}. Additionally, we have another program that runs the labeled chip-firing game at random as described in the conjecture. We have observed that when we use this program to play the game, we get the totally sorted configuration the most often. In fact, we have observed that configurations that correspond to SYT arise more frequently than those that do not.
    
    \begin{figure}[ht]
        \centering
        \[\begin{array}{c|c|c|c}
            && \#\text{ of stabilization sequences} & \text{total }\#\text{ of} \\
            (k,m) & \text{stable configuration} & \text{to config.} & \text{stabilization sequences} \\
            \hline
            \hline
            (1,1) & \BLUE{[1]} & 1 & 1 \\
            \hline
            (1,2) & \BLUE{[1,2]} & 2 & 2 \\
            \hline
            (1,3) & \BLUE{[1,2,3]} & 60 & 60 \\
            \hline
            \hline
            (2,1) & \BLUE{[1],[2]} & 1 & 1 \\
            \hline
            \multirow{2}{*}{(2,2)} & [1,3],[2,4] & 4 & \multirow{2}{*}{12} \\
            & \BLUE{[1,2],[3,4]} & 8 \\
            \hline
            \multirow{5}{*}{(2,3)} & [1, 3, 5], [2, 4, 6] & 8,568 & \multirow{5}{*}{179,424} \\
            & [1, 2, 5], [3, 4, 6] & 22,680 \\
            & [1, 3, 4], [2, 5, 6] & 22,680 \\
            & [1, 2, 4], [3, 5, 6] & 51,408 \\
            & \BLUE{[1, 2, 3], [4, 5, 6]} & 74,088 \\
            \hline
            \hline
            (3,1) & \BLUE{[1],[2],[3]} & 1 & 1 \\
            \hline
            \multirow{5}{*}{(3,2)} & [1, 4], [2, 5], [3, 6] & 12 & \multirow{5}{*}{120} \\
            & [1, 3], [2, 5], [4, 6] & 12 \\
            & [1, 3], [2, 4], [5, 6] & 24 \\
            & [1, 2], [3, 5], [4, 6] & 24 \\
            & \BLUE{[1, 2], [3, 4], [5, 6]} & 48 \\
        \end{array}\]
        \caption{\justifying Stable configurations reachable from $\Delta^{k,m}$ for selected small $(k,m)$ pairs, as well as the number of stabilization sequences for each configuration. Configurations are read so that the first set of brackets indicates chips on branch 1 ordered from the center outwards, the second set of brackets indicates those on branch 2, and so on. Totally sorted configurations are highlighted in blue.}
        \label{fig:probabilities}
    \end{figure}

    \FloatBarrier

\emergencystretch 2em
\printbibliography[]

\smallskip

\noindent
Annika Gonzalez-Zugasti \\
\textsc{
School of Mathematics, University of Minnesoata - Twin Cities\\
Vincent Hall 206 Church St. SE Minneapolis, MN 55455}\\
\textit{E-mail address: }\href{mailto:gonz0800@umn.edu}{\texttt{gonz0800@umn.edu}}
\medskip

\noindent
Ryan Lynch \\
\textsc{
School of Mathematics, University of Minnesoata - Twin Cities\\
Vincent Hall 206 Church St. SE Minneapolis, MN 55455}\\
\textit{E-mail address: }\href{mailto:rlynch@umn.edu}{\texttt{rlynch@umn.edu}}
\medskip

\noindent
Dylan Snustad \\
\textsc{
School of Mathematical and Statistical Sciences, Arizona State University\\
901 S Palm Walk, Tempe, AZ 85287}\\
\textit{E-mail address: }\href{mailto:dsnustad@asu.edu}{\texttt{dsnustad@asu.edu}}
\medskip

\end{document}